\documentclass{article}
\usepackage{amsmath, amsthm, amssymb}
\usepackage[all]{xy}
\usepackage[colorlinks]{hyperref}
\usepackage{mathrsfs}
\usepackage{enumerate}
\DeclareFontEncoding{OT2}{}{} 

\theoremstyle{plain}

\newtheorem{theorem}{Theorem}[section]

\newtheorem{proposition}[theorem]{Proposition}
\newtheorem{lemma}[theorem]{Lemma}

\newtheorem*{question*}{Question}

\theoremstyle{definition} 
\newtheorem{definition}[theorem]{Definition}

\theoremstyle{remark}

\renewcommand{\emptyset}{\varnothing}

\renewcommand{\phi}{\varphi}

\newcommand{\Kb}{\overline{K}}

\newcommand{\R}{\ensuremath{\mathbb{R}}}

\newcommand{\F}{\ensuremath{\mathbb{F}}}

\newcommand{\sm}{\ensuremath{\mathfrak{m}}}

\newcommand{\Q}{\ensuremath{\mathbb{Q}}}
\newcommand{\Z}{\ensuremath{\mathbb{Z}}}

\newcommand{\sL}{\ensuremath{\mathscr{L}}}
\newcommand{\sK}{\ensuremath{\mathscr{K}}}

\DeclareMathOperator{\Div}{Div}
\DeclareMathOperator{\dv}{div}
\DeclareMathOperator{\Pic}{Pic}

\DeclareMathOperator{\Hp}{H}

\newcommand{\rY}{Y_{\F_q}}

\begin{document}

\title{Period and index for higher genus curves}
\author{Shahed Sharif}
\maketitle

\section{Period and index}
\label{sec:period-index}

Let $K$ be a field; usually, we will consider it to be a number field. By a \emph{curve} $C$ over $K$, we shall mean a smooth projective geometrically integral curve. The \emph{period} and \emph{index} of $C$ over $K$ are two integer invariants which measure the failure of $C$ to have rational points. Specifically, the index is the gcd of degrees of all effective divisors $D \in \Div C$---that is, effective divisors which are rational over $K$. Equivalently, the index is the gcd of degrees $[L:K]$, where $L/K$ ranges over algebraic extensions such that $C(L) \neq \emptyset$. To see this, note that if $P \in C(L)$, then $D = \sum \sigma P$ is a rational effective divisor, where $\sigma$ ranges over the embeddings of $L$ into an algebraic closure of $K$; and conversely any minimal rational effective divisor is of this form. Also observe that if $C$ already has $K$-points, then its index is $1$.

The period is less stringent: here we look at the smallest positive degree of rational divisor \emph{classes}; these are given by divisors which are linearly equivalent to their Galois conjugates. To see that the two invariants need not be the same, consider any conic over $\R$ without rational points, say the curve with affine piece $x^2 + y^2 = -1$. Certainly the index is $2$, but as our curve is genus $0$, there is a single divisor class of degree $1$; namely, the class of a point. Therefore this class must be rational over $\R$, and so the period is $1$.

As a rational divisor automatically belongs to a rational divisor class, we have $P \mid I$. Furthermore, the canonical class gives a rational divisor of degree $2(g-1)$, so $I \mid 2(g-1)$. Lichtenbaum in~\cite{lic69} found further conditions on the possible values of the period and index:
\begin{theorem}[Theorem 8, \cite{lic69}]
  Let $C/K$ be a curve over a field with genus $g$, period $P$, and index $I$. Then
  \begin{enumerate}[(i)]
      \item $P \mid I \mid 2P^2$, and
      \item if either $2(g-1)/I$ or $P$ is even, then $I \mid P^2$.
  \end{enumerate}
\end{theorem}
We say a triple of integers $(g, P, I)$ is \emph{admissible} if they satisfy the divisibility conditions of Lichtenbaum's theorem; that is, if they are possible values for the genus, period, and index of a curve. The goal of this paper is to determine whether every admissible triple indeed occurs as the invariants of some curve over a number field. Our main result is the following:
\begin{theorem} \label{highergenus}
  Given any admissible triple $(g, P, I)$ such that $4\nmid I$, there exists a number field $K$ and a curve $Y$ over $K$ with genus $g$, period $P$ and index $I$.
\end{theorem}
Similar results have been proven earlier, including a complete answer in the genus $1$ case:
\begin{theorem}[Theorem 2, \cite{sha12}] \label{genus1K}
  Let $(1,P,I)$ be an admissible triple. Let $E$ be an elliptic curve over a number field $K$. Then there exists a genus 1 curve $X$ which is a principal homogeneous space for $E$ with period $P$ and index $I$.
\end{theorem}
A slightly weaker version of the above was proved in~\cite{CS10}, and was used to construct Shafarevich-Tate groups with arbitrarily high $p$-rank. 

Over other fields there are a scattering of results. In the case of $p$-adic fields, Lichtenbaum found stricter divisibility conditions:
\begin{theorem}[Theorem 7, \cite{lic69}]
  If $C$ is a curve over a finite extension of $\Q_p$ with genus $g$, period $P$, and index $I$, then $P \mid (g-1)$, $I \mid 2P$, and $I = P$ if $(g-1)/P$ is even.
\end{theorem}
For results on period and index over more general fields, see~\cite{CL15}. Applications of the period and index pop up typically in relation to the size of Shafarevich-Tate groups; see for example~\cite{PS99} and \cite{LLR04}.

\section{Outline of proof}

The method of proof is as follows. We will first construct a genus one curve $X$ with period $P(X)$ and index $I(X)$, and a degree $mI(X)$ map $\phi: X \to E''$ where $E''$ is some elliptic curve isogenous to the Jacobian of $X$ and $m$ is some integer. Then we will construct a smooth, projective genus 2 curve $Y'$ with period $P(Y') = 1$ and index $I(Y') = 1$ or $2$, and equipped with a certain degree $2$ map $\psi$ to $E''$. Then the curve in the theorem will be given by the fiber product $Y$ in the diagram
\[
  \label{square}
  \xymatrix{Y \ar[r] \ar[d] & X \ar[d]^\phi \\ Y' \ar[r]_\psi & E''}.
\]
As the genera of $X$ and $E''$ are both $1$, $\phi$ must be a finite \'etale morphism, and therefore the projection $Y \to Y'$ is finite \'etale as well. This implies that $Y$ is smooth and projective. Applying the Hurwitz formula with respect to the map $Y\to Y'$, the genus of $Y$ is $mI(X)+1$. We will show that, with some extra conditions, the period and index of $Y$ are $P(Y')P(X)$ and $I(Y')I(X)$, respectively.

This method is essentially the same as that used in the local case in~\cite{sha07}. However, the situation in the local case is easier and much more explicit, since we are able to make use of Tate curves. Furthermore, the set of admissible triples is much bigger in the number field case than in the local case.

A key fact is the following lemma. 
\begin{lemma} \label{mapdiv}
  Let $K$ be a field. If $\phi\colon C \to C'$ is a $K$-rational degree $d$ map of curves, then 
  \begin{enumerate}[(a)]
      \item $P(C') | P(C)$ and $I(C') | I(C)$; and
      \item $P(C) | dP(C')$ and $I(C) | dI(C')$.
  \end{enumerate}
\end{lemma}
    
\begin{proof}
  Let us consider the claims on the periods. If $\sL$ is a $K$-rational divisor class on $C$, then $\phi_*\sL$ is a $K$-rational class on $C'$ with the same degree, proving $P(C') | P(C)$. On the other hand, $\phi^*$ multiplies degrees by $d$, proving $P(C) | dP(C')$. The conclusions for the indices follows from a similar argument on rational divisors.
\end{proof}

To obtain the map $X \to E''$ we will need the following result:
\begin{lemma}\label{lemma:phs-index-map}
  Suppose $X$ is a genus 1 curve with index $I$ and Jacobian $E$. Then there is a degree $I$ $K$-rational morphism $X \to E$.
\end{lemma}

\begin{proof}
  Let $D$ be a rational divisor on $X$ of degree $I$. By Riemann-Roch, we may assume that $D$ is effective. Since $D$ is a minimal effective divisor on $X$, it must be of the form $\sum \sigma P$, where $P \in X(L)$ for some degree $I$ field extension $L/K$ and  $\sigma$ ranges over the embeddings of $L$ into an algebraic closure $\Kb$ of $K$. Since $X(L) \neq \emptyset$, there exists an $L$-rational isomorphism $f_0: X \to E$. Define $f = \sum f_0^\sigma$; one sees that $f$ is the desired morphism.
\end{proof}

We will also need the following result from class field theory.
\begin{lemma} \label{dirichlet}
  If $C_\sm$ is the ray class group with modulus $\sm$ over $K$, and if $\lambda$ is a class in $C_\sm$, then there are infinitely many primes of $K$ whose class in $C_\sm$ is $\lambda$.
\end{lemma}

\begin{proof}
  This is a well-known generalization of Dirichlet's theorem on primes in arithmetic progressions. See for example~\cite[p.~83, Theorem~1.111]{koc97}.
\end{proof}
To prove the main theorem, we will first deal with curves of genus 2, then the higher genus case.

\section{Genus 2 curves}

In this section, our goal is to show
\begin{proposition}\label{genus2K}
  Let $E/K$ be an elliptic curve such that $E[2]\subset E(K)$ and let $(2,P,I)$ be an admissible triple. Then there is an at most biquadratic extension $K'$ of $K$ such that over $K'$, there exists a genus 2 curve $Y$ with period $P$ and index $I$, equipped with a degree~2 map $\psi\colon Y\to E$.
\end{proposition}
In the genus 2 case, Lichtenbaum's theorem states that there are three possibilities:
\begin{enumerate}
    \item $P(Y)=I(Y)=1$,
    \item $P(Y)=1$ and $I(Y)=2$, and
    \item $P(Y)=I(Y)=2$.
\end{enumerate}

Write $E$ as $y^2=x(x-b)(x-c)$ with $b,c$ algebraic integers in $K$---we can do this since $E[2]\subset E(K)$. 

\paragraph{Case 1: $P(Y)=I(Y)=1$.}
Let $Y$ be the curve with function field $K(E)(\sqrt{x-a})$ for any $a\in K$ not equal to $0, b$, or $c$. Letting $z^2 = x - a$, we obtain the affine model for $Y$
$$
y^2=(z^2+a)(z^2+a-b)(z^2+a-c).
$$
Then $Y$ has two rational points at infinity, and hence period and index 1. The map $\psi$ is the obvious one. In this case, the field $K'$ is equal to $K$.

\paragraph{Case 2: $P(Y)=I(Y)=2$.}
Let $T$ be the 2-torsion point $(0,0)\in E$. Then there is a canonical isogeny 
$$
\psi\colon E\rightarrow \widetilde{E}:=E/\{O, T\}
$$
and a descent pairing
\begin{equation} \label{descpair}
  \frac{E(K)}{\widehat{\psi} (\widetilde{E}(K))} \times E[\psi] \rightarrow \frac{K^*}{K^{*2}}
\end{equation}
The function $x\in K(E)$ satisfies $\dv (x)=2(T)-2(O)$. Also, the pairing above is given by $(P,T)\mapsto x(P) \pmod{K^{*2}}$ for $P\neq T, O$. See, for example, \cite[Proposition~X.4.9]{sil86}. Note that these facts hold if we base-extend to a finite extension of $K$.

For a valuation $v$ of $K$, let $K_v$ be the completion of $K$ at $v$. Assume all nonarchimedean valuations $v$ are normalized so that $v(K_v^*)=\Z$. The following method is taken from \cite{CTP00}, specifically Lemmas~2.2 and 2.4. Let $S$ be the subset of the primes of $K$ which is the union of the set of primes dividing $bc$, the set of primes for which the corresponding residue field has fewer than 17 elements, the primes lying over $2$, and the set of real primes. Let $\sm$ be the modulus supported precisely on $S$ and such that the multiplicity of the primes in $S$ is one, with the exception of primes lying over $(2)$. At primes $v$ lying over $(2)$, we specify that $v(\sm)=2v(2)+1$. We apply Lemma~\ref{dirichlet} with $\lambda = (1)$ to find a principal prime $(\pi)$ of $K$ satisfying $\pi\equiv 1 \pmod{\sm}$; the prime can be taken to be principal since the identity in any ray class group is a class of principal ideals. Choose $\alpha\in K^*$ which is integral and a nonsquare in the completion of $K$ at $(\pi)$.

Now we are ready to construct $Y$. Adjoin the function $z$ to the function field of $E$ given by $z^2=\pi x+\alpha$. Then $Y$ has affine model 
\begin{equation} \label{Y:(2,2,2)}
  y^2=\pi^{-3}(z^2-\alpha)(z^2-\alpha-\pi b)(z^2-\alpha-\pi c).
\end{equation}

\begin{definition}[\cite{PS99}, p.~14]
  Let $C$ be a genus $2$ curve over $K$. A place $v$ of $K$ is said to be \emph{deficient} if the index of $C_v := C \times K_v$ is $2$.
\end{definition}
Note that $C(K_v) \neq \emptyset$ for all but finitely many $v$, and so for these places the local period and index are $1$. Therefore the number of deficient places is finite.

Deficiency is important because it can tell us whether the period and index are equal or not:
\begin{lemma}
  Fix a curve $C$ of genus $g$ over $K$. If the number of deficient places is odd, then the period $P(C)$ does not divide $g-1$.
\end{lemma}

\begin{proof}
  Let $J$ be the Jacobian of $C$, and consider $V := \mathbf{Pic}^{g-1}_{C/K}$, the degree $g-1$ component of the Picard scheme of $C$. Observe that $V$ is a principal homogeneous space for $J$ which has a rational point if and only if $P(C) \mid (g-1)$. As a consequence of~\cite[Corollary 12]{PS99}, if the number of deficient places is odd, then $V$ is a nontrivial principal homogeneous space, and hence has no rational points. The claim follows.
\end{proof}
We will prove that $Y$ has period 2 by showing that there is exactly one place where it is deficient, in our case the place corresponding to $\pi$, and that $Y$ is not deficient elsewhere. 

Let $v$ be any place for which $v(\sm)>0$, including the real places. By our choice of $\pi \equiv 1 \pmod{\sm}$, the points at infinity are rational over $K_v$. (At real places, the congruence means $\pi$ is positive.) Therefore, none of these places are deficient.

Now we consider places $v$ such that $v(\pi\cdot\sm)=0$. If $v$ is archimedean, then it must be complex, and $Y$ must have $K_v$ points. If $v$ is non-archimedean, we consider the reduction $\rY$ of the curve in the residue field of $K_v$; here, we let $q$ be the order of the residue field.. Recall that the primes dividing $bc$ also divide $\sm$, so that $v(\pi b) = v(\pi c)=0$. Therefore at least one of $v(\alpha), v(\alpha+\pi b), v(\alpha+\pi c)$ is zero, and so $\rY$ is reduced. By the Weil bounds for singular curves~\cite[Corollary~2.4]{AP96},
\[
  \# \rY(\F_q) \geq q + 1 - 4\sqrt{q}.
\]
Our hypotheses on $\sm$ imply that $q \geq 17$, implying that $q + 1 - 4\sqrt{q} > 1$ and therefore $\# \rY(\F_q) \geq 2$. The curve $\rY$ has at most one singular point, at $y = z = 0$. (Note that $\sm$ contains even primes.) Thus $\rY(\F_q)$ has at least one smooth point, which we can lift via Hensel's lemma to obtain an element of $Y(K_v)$. This shows that the index of $Y_v$ is $1$, and $v$ is not deficient.

Finally we consider the case where $v$ corresponds to $(\pi)$. Since $E$ has good reduction at $v$ and $v(P)=0$, the theory of the pairing~\eqref{descpair} tells us that $2\mid v(x(Q_E))$ for all $Q_E\in E(K_v)$ (see \cite[Theorem X.1.1]{sil86}). Suppose $Q\in Y(K_v)$. By inspection, neither of the points at infinity on $Y$ (with respect to the model~\eqref{Y:(2,2,2)}) lie in $Y(K_v)$, so we may assume $Q$ lies in the affine part of $Y$. Consider the right-hand side of equation~\eqref{Y:(2,2,2)}, which we will call $f(z)$. Let $Q_E$ be the image of $Q$ in $E(K_v)$, so that $z(Q)^2=\pi x(Q_E) + \alpha$. Our assumption that $\alpha$ is a nonsquare in $K_v$ implies that $x(Q_E)\neq 0$. If $v(x(Q_E))<0$, then since this valuation is even, $v(x(Q_E))<-1$. Therefore $2v(z(Q))=v(\pi x(Q_E))$; but the latter is odd. We conclude that $v(z(Q))$ lies in $\frac{1}{2}\Z \setminus \Z$, which contradicts our choice of $Q\in Y(K_v)$. If instead $v(x(Q_E))\geq 0$, then $2v(z(Q))=v(\alpha)=0$. Then $v(f(z(Q)))=-3$, implying $v(y(Q))=-3/2$. Once again, this is impossible. Therefore no such $Q$ can exist, so $Y(K_v)=\emptyset$. The descent pairing is functorial in $K$, so similar reasoning holds if we replace $K_v$ with any odd degree extension, which proves that the index of $Y$ must be even.

Therefore $Y$ is deficient at exactly one place of $K$, and has period 2 as desired.

Note that, as in the previous case, $K'=K$.

\paragraph{Case 3: $P(Y)=1$, $I(Y)=2$.}
The last case is only slightly more work. We use the same curve as in case 2, but now we base extend to 
\[
  K' = K\left(\sqrt{\frac{\alpha+\pi b}{\alpha}}, \sqrt{\frac{\alpha+\pi c}{\alpha}}\right).
\]
Both radicands are congruent to 1 $\pmod{(\pi)}$, and since $(\pi)$ does not lie above $(2)$, $(\pi)$ splits in this extension. Therefore we can repeat the reasoning above to show that, for $v$ a place of $K'$ lying over $(\pi)$, $Y(K'_v)$ is empty, and therefore the index of $Y$ is 2.

Consider the divisor on $\overline{Y} := Y \times \overline{K}$
\[
  D=(\sqrt{\alpha}, 0)+(\sqrt{\alpha+\pi b},0)+(\sqrt{\alpha+\pi c}, 0).
\]
The elements $\sqrt{\alpha}$, $\sqrt{\alpha+\pi b}$ and $\sqrt{\alpha + \pi c}$ each generate the same quadratic extension $K''$ of $K'$. The unique nontrivial automorphism $\sigma$ of $K''$ over $K'$ sends each square root to its additive inverse, so that the unique nontrivial conjugate of the divisor above is 
\[
  \sigma D =(-\sqrt{\alpha}, 0)+(-\sqrt{\alpha+\pi b},0)+(-\sqrt{\alpha+\pi c}, 0).
\]
Let $\sK$ be the divisor at infinity, i.e. the sum of the two points at infinity. (The notation comes from the fact that $\sK$ lies in the canonical class of $\overline{Y}$.) The function $z-\sqrt{\alpha}$ has divisor $2(\sqrt{\alpha}, 0)-\sK$, and similarly for $z-\sqrt{\alpha+\pi b}$ and $z-\sqrt{\alpha+\pi c}$. The function $y$ has divisor $D+\sigma D - 3\sK$. Therefore 
\[
  \dv \left(\frac{y}{(z-\sqrt{\alpha})(z-\sqrt{\alpha+\pi b})(z-\sqrt{\alpha+\pi c})}\right) = \sigma D - D
\]  
which shows $\sigma D$ is linearly equivalent to $D$. This proves that the class of $D$ is a degree 3 element of $\Hp^0(K', \Pic \overline{Y})$. Therefore the period of $Y$ equals $1$.

\section{Proof of Theorem~\ref{highergenus}}
\label{sec:proof-theor-refm}

We may assume that $g \geq 3$. The proof will be split into three cases:
\begin{enumerate}[(i)]
    \item $2 \nmid P$ and $P \mid I^2$,
    \item $2 \nmid P$ and $P \nmid I^2$, and
    \item $2 \mid P$.
\end{enumerate}
But before tackling these cases, we will conduct some general set-up. Consider an admissible triple $(g,P,I)$ with $4\nmid I$. If $2\mid I$ write $\widehat{I}=I/2$. If $2 \nmid I$, let $\widehat{I}=I$. Similarly if $P$ is even, let $\widehat{P}=P/2$, and if $P$ is odd $\widehat{P}=P$. Notice that $\widehat{I}$ must divide $(g-1)$ since $I$ divides $(2g-2)$. Let $m$ be $(g-1)/\widehat{I}$. We willalso need the following ingredients.
\begin{itemize}
    \item An elliptic curve $E$ over a number field $K$ satisfying $E[\widehat{P}]\subset E(K)$ and, if $\widehat{P}$ is even, $E[4]\subset E(K)$.
    \item An isogeny $\phi_0\colon E\to E'$ with kernel isomorphic to $\Z/\widehat{P}\Z \times \Z/\ell\Z$, where $\ell = \widehat{I}/\widehat{P}$. 
    \item An isogeny of degree $m$, $E'\to E''$, where $E''$ is an elliptic curve satisfying $E''[2]\subset E(K)$.
\end{itemize}
In case (i), we will additionally need the Mordell-Weil rank of $E'$ to be at least 1. In case (ii), we may need to replace $K$ with a quadratic or biquadratic extension of $K$. Regardless, these conditions are easy to satisfy: take any elliptic curve over $\Q$ for $E$, say, then let $K$ be a sufficiently large finite extension over which all of the conditions hold.

Next we apply Proposition~\ref{genus2K} to get a genus 2 curve $Y'$ equipped with a degree 2 map $\psi\colon Y'\to E''$. In case (i), we choose $Y'$ so that $P(Y')=I(Y')=1$. In case (ii), we choose $Y'$ so that $P(Y')=1$ and $I(Y')=2$---we may need to replace $K$ with a quadratic or biquadratic extension in order to construct $Y'$, which we do now. In case (iii), we choose $Y'$ so that $P(Y')=I(Y')=2$.

Next apply Theorem~\ref{genus1K} to find a genus 1 curve $X$ with Jacobian $E'$, period $\widehat{P}$, and index $\widehat{I}$. We would like to construct a rational degree $m\widehat{I}$ map $\phi\colon X\to E''$. We do this by first constructing a degree $\widehat{I}$ map $f:X\to E'$ as in Lemma~\ref{lemma:phs-index-map}, which we compose with the isogeny $\phi_0$.

Finally construct the fiber product in diagram~\eqref{square}, which we recall here:
$$
\xymatrix{Y \ar[r] \ar[d] & X \ar[d]^\phi \\ Y' \ar[r]_\psi & E''.}
$$
The following facts are a direct consequence of the diagram and Lemma~\ref{mapdiv}:
\begin{enumerate}
    \item $P(X)\mid P(Y)\mid 2P(X)$ 
    \item $I(X)\mid I(Y)\mid 2I(X)$
    \item $P(Y')\mid P(Y)\mid mI(X)P(Y')$
    \item $I(Y')\mid I(Y)\mid mI(X)P(Y')$ 
\end{enumerate}
Since the map $\phi$ is a map between genus 1 curves, it must be unramified. Therefore the left hand map in the fiber product diagram, $Y\to Y'$, is also unramified. By the Riemann-Hurwitz formula, we have
$$
2g(Y)-2=m\widehat{I}(2g(Y')-2)
$$
from which $g(Y)-1=m\widehat{I}$. Therefore $g(Y)=g$. It remains to show that $P(Y)=P$ and $I(Y)=I$.

\paragraph{Case i: $2\nmid P, I\mid P^2$.}
There is one further specification that must be made in the set-up above. In the construction of $Y'$, which should have period and index 1, recall from the proof of Proposition~\ref{genus2K} that we formed the function field of $Y'$ by adjoining the square root of the function $x-a$ to the function field of $E''$ for any $a \neq 0, b, c$. (A Weierstrass model of $E''$, $y^2=x(x-b)(x-c)$, was given.) For this case, we must specify the choice of $a$. By hypothesis $E'$ has positive rank. Let $Q'\in E'(K)$ have infinite order. Let $Q''$ be the image of $Q'$ under the degree $m$ isogeny $E'\to E''$. Then choose $a=x(Q'')$. Since $Q''$ has infinite order, $a\neq 0,b,c$. Additionally, $Q''$ is a ramification point of $\psi$.

The divisibility conditions in this case imply that $P$ and $I$ are odd, and therefore $\widehat{P}=P$ and $\widehat{I}=I$. The set-up tells us that $P(X)=P$, and $I(X)=I$. Also, according to the set-up, we have chosen $Y'$ to satisfy $P(Y')=I(Y')=1$. 

We now prove that $I(Y)=I$. Let $[Q']$ denote the prime divisor associated to $Q'$. Taking the pullback of $[Q']$ under $X\to E'$ gives a rational divisor of degree $I$. Pulling back to $Y$ gives a rational divisor $D$ of degree $2I$. 

But $\psi$ is ramified over $Q''$. Since $\phi$ is an unramified map, the projection $Y\rightarrow X$ is ramified over every point of $D$; i.e., $D=2D_0$ for some rational divisor $D_0$. Then $D_0$ has degree $I$.

From the divisibility conditions in the last section, $I\mid I(Y)$. Together with the above, we have $I(Y)=I$.

Since $P(Y)\mid I(Y)$ and $I(Y)=I$ is odd, then $P(Y)$ must be odd as well. But $P\mid P(Y)\mid 2P$, so $P(Y)=P$ also.

\paragraph{Case ii: $2\nmid P, I\nmid P^2$.}

By Lichtenbaum's theorem, $I\mid 2P^2$, and hence $I$ is even. The general set-up tells us that $P(X)=\widehat{P}=P$ and $I(X)=\widehat{I}=I/2$. Since $4\nmid I$, $I(X)$ is odd. Also in the set-up, we chose $Y'$ so that $P(Y')=1$ and $I(Y')=2$.

Recall $m = (g-1)/\widehat{I}$. Since $\widehat{I}=I/2$, we have $m=(2g-2)/I$. As $(g,P,I)$ is admissible, $m$ must be \emph{odd}, or else by Lichtenbaum's theorem we would have $I\mid P^2$. At this point, just writing out the divisibility conditions gives the result: $P(X)\mid P(Y)\mid mI(X)P(Y')=mI(X)$, which is odd, and $P(Y)\mid 2P(X)$, and therefore $P(Y)=P(X)=P$. Also, $I(X)\mid I(Y)\mid 2I(X)$ and $2=I(Y') \mid I(Y)$, and hence $I(Y)=2I(X)=I$.

\paragraph{Case iii: $2\mid P$.}

In this case, both $P$ and $I$ are even. Therefore $\widehat{P}=P/2$ and $\widehat{I}=I/2$. From the set-up, we have $P(X)=P/2$, $I(X)=I/2$, $P(Y')=2$, and $I(Y')=2$. Also, the hypotheses of the theorem require $4\nmid I$.

According to Lichtenbaum's theorem, $2\mid P$ implies that $I\mid P^2$. Now $I(X)$ must be odd, and since $P(X) | I(X)$, $P(X)$ must also be odd. The divisibility conditions immediately imply $P(Y)=P$ and $I(Y)=I$.

This completes the proof of Theorem~\ref{highergenus}.

\bibliographystyle{alpha}
\end{document}